\documentclass[11pt]{amsart}
\usepackage{amsmath}
\usepackage{amssymb}
\usepackage{amsfonts}
\usepackage{latexsym}
\usepackage{fancyhdr}
\usepackage{graphics}
\usepackage{graphicx}
\usepackage{color}
\newtheorem{defn}{Definition}[section]
\newtheorem{prop}[defn]{Proposition}

\newtheorem{cor}[defn]{Corollary}

\theoremstyle{definition}
\newtheorem{rem}[defn]{Remark}
\def\x{\relax\ifmmode {\mbox{*}}\else*\fi}
\newcommand{\bedefi}{\begin{defn}$\!\!${\bf }$\;$\rm }
\newcommand{\findefi}{\end{defn}}

\newcommand{\mc}{\mathcal}

\newcommand{\M}{\mc M}

\newcommand{\D}{\mc D}

\def\H{\mc H}

\newcommand{\LDD}{{\mc L}(\D,\D^\times)}
\newcommand{\ip}[2]{\left\langle {#1}|{#2}\right\rangle}
\begin{document}
\title[Riesz-Fischer maps in rigged Hilbert spaces]{Lower bounds for Riesz-Fischer maps in rigged Hilbert spaces}%

\author{Francesco Tschinke}%
\address{Dipartimento di Matematica e Informatica, Universit\`a degli Studi
di Palermo, I-90123 Palermo (Italy)} \email{francesco.tschinke@unipa.it}
\begin{abstract}
\noindent
This note concerns a further study about Riesz-Fischer maps, already introduced by the author in a recent work, that is a notion
 that extends to the spaces of distributions the sequences that are known as Riesz-Fischer sequences.
In particular it is proved a characterizing inequality that have as consequence the existence of the continuous inverse of the synthesis operator.
\end{abstract}
\maketitle
\section{Introduction}
As known, given an element $f\in\H$ and a sequence of elements $\{f_n\}_{n\in\mathbb N}$ in a Hilbert space $\H$ endowed of the inner product $\ip{\cdot}{\cdot}$,  the sequence $a_n:\mathbb N\rightarrow\mathbb C$, $a_n:=\ip{f}{f_n}$ is called {\it moment sequence} -briefly \textit{moments}- \textit{of} $f\in \H$. The problem to find a solution $f\in\H$ of the equations:
$$
\ip{f}{f_n}=a_n,\quad n\in\mathbb N
$$
given   $\{f_n\}_{n\in\mathbb N}$ and $\{a_n\}_{n\in\mathbb N}\subset\mathbb C$, is known as \textit{ moment problem}.

In particular, the sequence $\{f_n\}_{n\in\mathbb N}$ is called \textit{Riesz-Fischer sequence} if, for every $\{a_n\}_{n\in\mathbb N}\in l^2$ (i.e. such that $\sum_1^\infty|a_n|^2<\infty$), there exists a solution $f$ of the moment problem.

On the other hand, the sequence $\{f_n\}_{n\in\mathbb N}$ is called \textit{Bessel sequence} if, for all $f\in\H$, one has $\ip{f}{f_n}_{n\in\mathbb N}\in l^2$.

 We have the well-known characterization results \cite[Th. 3, Sec. 2]{young}:
 \begin{itemize}
 \item
 $\{f_n\}_{n\in\mathbb N}$ is a Riesz-Fischer sequence if, and only if, there exists $A>0$ such that:
\begin{equation}
\label{cararf}
A\sum_{n=1}^k |c_n|^2\leq \sum_{n=1}^k\|c_n f_n\|^2
\end{equation}
for all finite scalar sequences $\{c_n\}\subset\mathbb C$;
 \item
 $\{f_n\}_{n\in\mathbb N}$ is a Bessel sequence if, and only if, there exists $B>0$ such that:
\begin{equation}
\label{carabess}
\sum_{n=1}^k\|c_n f_n\|^2\leq B\sum_{n=1}^k|c_n|^2
\end{equation}
for all finite scalar sequences $\{c_n\}\subset\mathbb C$.
 \end{itemize}
Bessel and Riesz-Fischer sequences plays an important role   in the theory of frames and, in particular, in the study of Riesz bases \cite{young, christensen2}. Roughly speaking, a frame is an extension of a basis in a Hilbert space, in the sense that every vector of $\H$ can be decomposed in terms of elements of a frame, but this decomposition is not in unique.  This "loss of constraints" or "more leeway"  allows several applications in many branches of mathematical sciences and technology.

More precisely, a sequence $\{f_n\}_{n\in\mathbb N}$ in $\H$ is a \textit{frame} if there exists $A,B>0$ such that:
$$
A\|f\|^2\leq \sum_{n\in\mathbb N}|\ip{f}{f_n}|^2\leq B\|f\|^2, \forall f\in\H .
$$
A frame   $\{f_n\}_{n\in\mathbb N}$ that is also a basis for $\H$, is called \textit{Riesz basis}. Furthermore, Bessel, Riesz-Fischer sequences and Riesz basis are related via linear operators to orthonormal basis  (see \cite{young, christensen2, balastoeva}). Furthermore, if  $\{f_n\}_{n\in\mathbb N}$ is complete  or total (i.e. the set of its linear span is dense in $\H$), then it is a Riesz basis if, and only if, it is both Bessel and Riesz-Fischer sequence \cite{young}.

If it is not diversely specified, a frame is intended as discrete. However, a notion of \textit{continuous frame} have been introduced by \cite{keiser} and \cite{AAG_book, AAG_paper}, the last in order to  study coherent states. Instead of a sequence, it is considered a map $F$ from a measure space $(X,\mu)$ ($\mu$ is a positive measure) and a Hilbert space $\H$, i.e.: $F: X\rightarrow \H$, $F: x\mapsto F_x$. This map is called continuous frame with respect to $(X,\mu)$ if:
\begin{itemize}
	\item
	$F$ is weakly measurable, i.e. that is $x\rightarrow \ip{f}{F_x}$ is $\mu$-measurable for every $f\in \H$;
	\item
	there exists $A, B>0$ such that:
	$$
	A\|f\|^2\leq \int_X|\ip{f}{F_x}|^2d\mu\leq B\|f\|^2, \quad\forall f\in\H .
	$$
\end{itemize}

In \cite{TTT}, C. Trapani, S. Triolo  and the author have extended  the notion of frames, and related topics as Bessel, bases, Riesz basis, etc., to the spaces of distributions; in \cite{FT} and \cite{corsotsch} a further study  has be done respectively for Riesz-Fischer sequences and multipliers.

An appropriate framework  for the spaces of distributions is given by the \textit{rigged Hilbert space} i.e. is a triple $\D\subset\H\subset \D^\times$  where $\D$ is a locally convex space, $\D^\times$ the conjugate dual, and where the inclusions have to be intended as continuous and dense embedding.
They have been introduced by Gel'fand  in \cite{gelf3, gelf}   with the aim to define the \textit{generalized eigenvectors} of a essentially self-adjoint operator on $\D$ and to prove the  theorem known as Gel'fan-Maurin theorem, on the existence of a complete system of generalized eigenvector (see also \cite{gould}). For that   it is called also \textit{Gel'fand triple}, also denoted by $(\D,\H,\D)$.

However \textit{Gel'fand triple} plays a relevant role also in other branches of mathematics, such as  Gabor Analysis: see for example \cite{cordero,feich2,feich 3}. But you can find in the site www.nuhag.eu/talks  more exhaustive references, papers and talks.

 Reconsidering Riesz-Fischer maps  introduced in \cite{FT}, the aim of this paper is to continue the study, proving characterizing conditions,  analogously as for the Riesz-Fischer sequences that are characterized by the inequality (\ref{cararf}). The paper is organized as follows. In Section 2  are recalled some preliminaries, definitions and previous results. In Section 3 are proved some characterizations of Riesz-Fischer maps  in term of lower bounds properties.
\section{Preliminary definitions and facts}
As usual, let us denote as $\H$ a Hilbert space, $\ip{\cdot}{\cdot}$ is the inner product and $\|\cdot\|$ the Hilbert norm.
Let $\D$ be a  dense subspace of $\H$ endowed with a locally convex topology $t$ stronger than the topology induced by the Hilbert norm. The embedding of $\D$ in $\H$ is continuous and dense, and it is denoted by $\D\hookrightarrow\H$. The space of conjugate linear continuous forms on $\D$ is called  \textit{conjugate dual of} $\D$ and it is denoted by $\D^\times$.
Unless otherwise stated, the value of $F\in\D^\times$ on $f\in\D$ is denoted by $\ip{f}{F}$. The space $\D^\times$ is endowed with the \textit{strong dual topology} $t^\times=\beta(\D^\times,\D)$ defined by the set of seminorms:
\begin{equation}\label{semin_Dtimes}
p_\M(F)=\sup_{g\in \M}|\ip{g}{F}|, \quad F\in \D^\times,
\end{equation}
where $\M$ is a bounded subset of $\D[t]$. In this way, the Hilbert space $\H$ can be continuously embedded as subspace of $\D^\times$ (see \cite{horvath}). If $\D$ is reflexive, i.e. $\D^{\times\times}=\D$, the embedding is dense. We obtain the Gel'fand triple:
\begin{equation}\label{eq_one_intr}
\D[t] \hookrightarrow  \H \hookrightarrow\D^\times[t^\times],
\end{equation}
where $\hookrightarrow$ denote a continuous and dense embedding.
The sesquilinear form $\ip{\cdot}{\cdot}$ that put  $\D$ and $\D^\times$ in duality  is an extension of the inner product of
$\H$ and the notation is the same. We put: $\ip{F}{f}:=\overline{\ip{f}{F}}$.

Through the paper, $(X,\mu)$ a measure space, where $\mu$ is a $\sigma$-finite positive measure. We write $L^1(X,\mu), L^2(X,\mu),\cdots$ the usual spaces of measurable functions. In the case $X=\mathbb R$,  and $\mu$ is the Lebesgue measure, we denote them as $L^p(\mathbb R)$. Furthermore,  $\mathcal S$ stands as the \textit{Schwartz space}, i.e. the space of infinitely differentiable and rapidly decreasing functions on $\mathbb R$. The conjugate dual of $\mathcal S$,  denoted by $\mathcal S^\times$,  is known as the space of \textit{tempered distributions} (see  \cite{reed1} for  more accurate definitions). An usual example of rigged Hilbert space is given by:
$$
\mathcal S\hookrightarrow L^2(\mathbb R)\hookrightarrow\mathcal S^\times.
$$
The vector space of all continuous linear maps from $\D[t]$ into  $\D^\times[t^\times]$ is denoted by $\LDD$. {If $\D[t]$ is barreled (e.g.~reflexive)},   can be introduced an involution in $\LDD$,  $X \mapsto X^\dag$, by:
\begin{equation}
\label{eq: X dag}\ip{X^\dag \eta}{ \xi} = \overline{\ip{X\xi}{\eta}}, \quad \forall \xi, \eta \in \D.
\end{equation}
Hence, in this case, $\LDD$ is a $^\dagger$-invariant vector space.\\

In this paper we consider maps with values in a distribution spaces,  defined in \cite{TTT} as \textit{weakly measurable maps}, and here denoted by $\omega$. The definition extends the notion, previously recalled  in the introduction, of weakly measurable functions considered for continuous frame (see \cite{AAG_paper}).
\bedefi
The correspondence $\omega:X\rightarrow \D^\times$, $x\mapsto \omega_x$ is called \textit{weakly measurable map} if the complex valued function $x\mapsto\ip{f}{\omega_x}\in\mathbb C$ is $\mu$-measurable for all $f\in\D$.
\findefi
In particular, the notions of completeness and independence of sequences in the Hilbert space is extended to the spaces of distributions by the following:
\bedefi
\label{tandg}
 Let $\omega: x\in X\to \omega_x \in \D^\times$ be a weakly measurable map, then:
\begin{itemize}
\item[i)] $\omega$ is \textit{total} or {\it complete} if, $f \in \D $ and $\ip{f}{\omega_x}=0$  $\mu$-a.e. $x \in X$ implies $f=0$;
\item[ii)]$\omega$ is \textit{$\mu$-independent} if the unique measurable function $\xi:X\rightarrow \mathbb C$ such that  $\int_X \xi(x)\ip{g}{\omega_x} d\mu=0$ for every $g \in \D$, is $\xi(x)=0$ $\mu$-a.e.
\end{itemize}
\findefi
Let us recall the notion of \textit{Bessel map}:
\bedefi
A weakly measurable map $\omega$ is a {\em Bessel distribution map} (briefly: Bessel map) if for every $f \in \D$,
$ \int_X |\ip{f}{\omega_x}|^2d\mu<\infty$.
\findefi
As a consequence of the closed graph theorem, if $\D$ is a Fr\`echet space, for the Bessel maps one has the following characterization result:
\begin{prop}[{\cite[Proposition 3.1]{TTT}}]\label{prop2} If $\D[t]$ a Fr\`{e}chet space, and $\omega: x\in X \to \omega_x\in \D^\times$ a weakly measurable map. The following statements are equivalent.
\begin{itemize}
\item[(i)]
$\omega$ is a  Bessel  map;
\item[(ii)]there exists a continuous seminorm $p$ on $\D[t]$ such that:
\begin{equation*}\label{eqn_bessel1}\left( \int_X |\ip{f}{\omega_x}|^2d\mu\right)^{1/2}\leq p(f), \quad \forall f \in \D.\end{equation*}
 \item[(iii)] for every bounded subset $\mathcal M$ of $\D$ there exists $C_{\mathcal M}>0$ such that:
\begin{equation*}\label{eqn_bessel2}
\sup_{f\in\mathcal M}{\Bigr |}\int_X\xi(x)\ip{\omega_x}{f}d\mu{\Bigl |}\leq C_{\mathcal M}\|\xi\|_2, \quad \forall \xi\in L^2(X,\mu).
\end{equation*}
  \end{itemize}
\end{prop}
The previous proposition has the following consequences \cite{TTT}:
\begin{itemize}
\item
If $\xi\in L^2(X,\mu)$, then the  conjugate linear functional $ {\Lambda^\xi_\omega}$ on $\D$  defined by:
\begin{equation}
\label{functlambda}
\ip{f} {\Lambda^\xi_\omega}:=\int_X\xi(x)\ip{f}{\omega_x}d\mu,\quad\forall f\in\D
\end{equation}
is defined and continuous, i.e. $\Lambda_\omega^\xi\in \D^\times[t^\times]$;
\item
the {\em synthesis operator}  $D_\omega:L^2(X, \mu)\to \D^\times[t^\times]$  defined by
$ D_\omega: \xi \mapsto {\Lambda^\xi_\omega}$
 is continuous;
\item
the {\em analysis operator} $C_\omega: \D[t]\to L^2(X, \mu)$ defined by
$(C_\omega f)(x) =\ip{f}{\omega_x}$
is continuous;
\item
the {\it frame operator} $S_\omega:\D\rightarrow\D^\times$, $S_\omega:=D_\omega C_\omega$ is continuous, i.e. $S_\omega\in\LDD$.
\end{itemize}
\begin{rem}
If $\D$ is a Fr\`echet space, a Bessel map $\omega$ is upper bounded by a continuous seminorm on $\D$, but, in general, is not upper bounded by the Hilbert norm. An example, considered in \cite{FT}, is the system  of derivative of Dirac deltas on $\mathcal S\hookrightarrow L^2(\mathbb R)\hookrightarrow \mathcal S^\times$ denoted by $\{\delta'_x\}_{x\in\mathbb R}$, and defined by: $\ip{f}{\delta'_x}:=-f'(x)$ . Then $\delta'_x$  is a Bessel map but is not upper  bounded by the Hilbert norm. In \cite{TTT} is defined the \textit{bounded Bessel map}, i.e. a Bessel map $\omega$ such that there exists $B>0$ such that: $\int_X|\ip{f}{\omega_x}|^2d\mu\leq B\|f\|^2$ for all $f\in\D$. In particular, if $\omega$ also total and if there exists $B>0$ such that $0<\int_X|\ip{f}{\omega_x}|^2d\mu\leq B\|f\|^2$ forall $f\neq 0$,   then $\omega$ it is called  \textit{distribution upper semiframe} \cite{FT},  as extension to the space of distributions of the corresponding notion of \textit{continuous upper semiframe}  introduced in \cite{semifr1}.
\end{rem}
If a bounded Bessel map $\omega$ is also bounded from below by the Hilbert norm, we have the definition of \textit{distribution frame}:
\bedefi\cite[Definition 3.6]{TTT} \label{defn_distribframe}  Let $\D[t] \hookrightarrow\H\hookrightarrow\D^\times[t^\times]$ be a rigged Hilbert space, with $\D[t]$ a reflexive space and $\omega$ a Bessel map.
We say that $\omega$ is a {\em  distribution frame} if there exists $A,B>0$ such that:
\begin{equation*}
\label{eqn_frame_main1}
A\|f\|^2 \leq \int_X|\ip{f}{\omega_x}|^2d\mu \leq B \|f\|^2, \quad \forall f\in \D.
\end{equation*}
We have that (see \cite{TTT} for details): $\Lambda_\omega^\xi$ is bounded in $(\D,\|\cdot\|)$ and the bounded extension to $\H$ is denoted by $\widetilde{\Lambda}_\omega^\xi$; the synthesis operator $D_\omega$ has range in $\H$ and it is bounded; the Hilbert adjoint $D^*_\omega$ extends $C_\omega$ to $\H$; the  operator $\tilde{S}_\omega=D_\omega D_\omega^*$ is bounded and  extends the frame operator $S_\omega$.
\findefi
If $\omega$ is a distribution frame, then the frame operator $\hat{S}_\omega$  enjoys  the inequality
$$ A\|f\| \leq \| \hat{S}_\omega f\| \leq B\|f\|,\quad \forall f\in \H. $$
Since  $\hat{S}_\omega$ is symmetric, this  implies that $\hat{S}_\omega$ has a bounded inverse $\hat{S}_\omega^{-1}$ everywhere defined in $\H$.

In \cite{FT} are defined the Riesz-Fischer maps in the space of distributions. They are the analogous of the corresponding sequences in Hilbert space and where an extension to the continuous case is given in \cite{rahimi}.
\bedefi\cite[Definition 3.4]{FT}
Let $\D[t]$ be a locally convex space. A weakly measurable map $\omega: x\in X \mapsto \omega_x\in \D^\times$ is called a {\em Riesz-Fischer distribution map} (briefly: Riesz-Fischer map)  if,
for every $h\in L^2({X,\mu})$,
there exists $f \in \D$ such that:
\begin{equation}
\label{rf}
\ip{f}{\omega_x}=h(x)\quad \mbox{$\mu$-a.e.}
\end{equation}
In this case, we say  that $f$ is a solution of equation $\ip{f}{\omega_x}=h(x)$.
\findefi
Clearly, if $f_1$ and $f_2$ are solutions of (\ref{rf}), then $f_1-f_2\in\omega^\bot:=\{g\in \D: \ip{g}{\omega_x}=0, \quad \mu-a.e.\}$. If $\omega$ is total, the solution is unique.

The analysis operator $C_\omega$ is defined on $dom(C_\omega):=\{f\in\D: \ip{f}{\omega_x}\in L^2(X,\mu)\}$ as $C_\omega: f\in dom(C_\omega)\mapsto\ip{f}{\omega_x}\in L^2(X,\mu)$. Clearly,  $\omega$ is a Riesz-Fischer map if and only if $C_\omega: dom(C_\omega)\rightarrow L^2(X,\mu)$ is surjective. If $\omega$ is total, it is injective too, so, in this case, $C_\omega$ is invertible.

To define the synthesis operator $D_\omega$ we consider the following subset of $L^2(X,\mu)$:
$$dom(D_\omega):=\{\xi\in L^2(X,\mu), \mbox{s.t.} \, \int_X\xi(x){\omega_x}d\mu\, \mbox{is convergent in}\,\, \D^\times\}.$$
As \textit{convergent in} $\D^\times$ we mean: $\int_X\xi(x)\ip{f}{\omega_x}d\mu$ is convergent for all $f\in\D$ and the conjugate functional on $\D$ defined in (\ref{functlambda}) by  $\Lambda_\omega^\xi$, is continuous, so $\Lambda_\omega^\xi\in\D^\times$.
Then the synthesis operator  $D_\omega: dom(D_\omega)\rightarrow  \D^\times$ is defined by:
$$D_\omega: \xi\mapsto \Lambda_\omega^\xi:=\int_X\xi(x)\omega_xd\mu.$$
The range of $D_\omega$ is denoted by $Ran (D_\omega)$:
$$
Ran (D_\omega):=\Bigl{\{ }F\in\D^\times: \exists\, \xi\in dom(D_\omega):\, \forall f\in\D,\, \ip{f}{F}:=\int_X\xi(x)\ip{f}{\omega_x}d\mu\Bigr{\}}.
$$
If $\D$ is a Fr\`echet space, as a consequence of the closed graph theorem one has, for a total Riesz-Fischer map  $\omega$,  the following inequality holds:
\begin{cor}\cite[Corollary 3.7]{FT}
\label{lbound}
Assume that $\D[t]$ is a Fr\`echet space. If the map $\omega: x\in X \to \omega_x\in \D^\times$ is a total Riesz-Fischer map, then for every continuous seminorm $p$ on $\D$, there exists a constant $C>0$ such that, for the solution $f$ of \eqref{rf}:
$$
{p}({f})\leq C\|\ip{f}{\omega_x}\|_2.
$$
\end{cor}
It follows that, if $\omega$ a total Riesz-Fischer map, the inverse of the analysis operator $C_\omega^{-1}:L^2(X,\mu)\rightarrow\, dom(C_\omega)$ is continuous.
\section{Main results}
In this section are proved some characterization properties of Riesz-Fischer maps.  We have the following:
\begin{prop}
\label{propdisrf}
	Let $(X,\mu)$ be a measure space, $h(x)\in L^2(X,\mu)$  and $\omega: X\ni x\mapsto\omega_x\in\D^\times$ a weakly measurable map. Then $\omega$ is a Riesz-Fischer map if, and only if, there exists a bounded subset $\mathcal M\subset\D$ such that:
	\begin{equation}
	\label{disrf}
	{\Bigl |}\int_X\xi(x)\overline{h(x)}d\mu{\Bigl |}\leq\sup_{f\in\mathcal M}{\Bigl |}\int_X\xi(x)\ip{\omega_x}{f}d\mu{\Bigl |}
	\end{equation}
	for all $\xi(x)\in L^2(X,\mu)$ such that  $\int_X\xi(x){\omega_x}d\mu$ is convergent in $\D^\times$.
\end{prop}
\begin{proof}
	Necessity is obvious: let $\bar{f}$ be a solution of (\ref{rf}), then, for all $\xi(x)$, one has:
	$$
	{\Bigl |}\int_{X}\xi(x)\overline{h(x)}d\mu{\Bigl |}={\Bigl |}\int_{X}\xi(x)\ip{\omega_x}{\bar{f}}d\mu{\Bigl |}\leq\sup_{f\in\mathcal M}{\Bigl |}\int_{X}\xi(x)\ip{\omega_x}{f}d\mu{\Bigl |}.
	$$
	Sufficiency: let us consider the subspace $\mathcal E\subset\D^\times$  defined by the set of $F\in\D^\times$ such that there exist $\xi\in  L^2(X,\mu)$:
	$F=\int_X\xi(x)\omega_xd\mu$, and let us define the linear functional $\nu$ on $\mathcal E$ by:
	$$
	\nu(F)=\nu{\Bigl (}\int_{X}\xi(x)\omega_xd\mu{\Bigr )}:=\int_{X}\xi(x)\overline{h(x)}d\mu.
	$$
It follows immediately from hypothesis that $\nu$ is defined unambiguously.
	Furthermore, from (\ref{disrf}) one has:
	$$
	{\bigl |}\nu(F){\bigl |}={\Bigl |}\int_{X}\xi(x)\overline{h(x)}d\mu{\Bigl |}\leq\sup_{f\in\mathcal M}{\Bigl |}\int_{X}\xi(x)\ip{\omega_x}{f}d\mu{\Bigl |}=\sup_{f\in\mathcal M}|\ip{f}{F}|, \forall F\in \mathcal E
	$$
	i.e. $\nu(F)$ is bounded by a seminorm of $\D^\times[t^\times]$. By Hahn-Banach theorem, there exists an extension $\widetilde{\nu}$ of $\nu$ to $\D^\times$. Since $\D$ is reflexive,
	there exists $\bar f\in\D$ such that: $\widetilde{\nu}(F)=\ip{F}{\bar f}$. Since:
	$$
	\int_X\xi(x)[\ip{\omega_x}{\bar f}-\overline{h(x)}]d\mu=\nu(F)-\int_X \xi(x)\overline{h(x)}d\mu= 0,\quad \forall\, \xi\in L^2(X,\mu)
	$$
	then $\ip{\bar f}{\omega_x}=
	h(x)$ $\mu$-a.e.
\end{proof}
As consequence, we have the following:
\begin{cor}
\label{corodisRF}
	$\omega$ is a Riesz-Fischer map if, and only if, there exists a bounded subset ${\mathcal M}\subset\D$  such that:
	\begin{equation}
	\|\xi\|_2\leq \sup_{f\in\mathcal M}{\Bigl |}\int_X\xi(x)\ip{f}{\omega_x}d\mu{\Bigl |},
	\label{disRF}
	\end{equation}
	for all $\xi(x)\in L^2(X,\mu)$ such that $\int_X\xi(x){\omega_x}d\mu$ is convergent in $\D^\times$.
\end{cor}
\begin{proof}
	Sufficiency: the  condition (\ref{disRF}) holds. Let us consider $h(x)\in L^2(X,\mu)$ with $\|h(x)\|_2\leq 1$:
	$$
	{\Bigl |}\int_X\xi(x)\overline{h(x)}d\mu{\Bigl |}\leq\|\xi\|_2\leq
	\sup_{f\in\mathcal M}{\Bigl |}\int_X\xi(x)\ip{f}{\omega_x}d\mu{\Bigl |}.
	$$
	Then, for the previous proposition, $\omega$ is a Riesz-Fischer map.\\
	Necessity: since $\omega$ is a Riesz-Fischer map, putting $\frac{\xi(x)}{\|\xi(x)\|_2}=h(x)$, for the previous proposition there exists a bounded subset $\mathcal{M}\subset \D$ such that:
	$$
	\|\xi\|_2=\int_X\xi(x)\frac{\overline{\xi(x)}}{\|\xi(x\|_2}d\mu\leq
	\sup_{f\in\mathcal M}{\Bigl |}\int_X\xi(x)\ip{f}{\omega_x}d\mu{\Bigl |}.
	$$
\end{proof}
The previous Corollary can be rephrased as:
\begin{cor}
  $\omega$ is a   Riesz-Fischer map if, and only if, the synthesis operator $D_\omega$ is invertible and the inverse $D_\omega^{-1}: Ran ({D_\omega})\rightarrow L^2(X,\mu)$ is continuous.
\end{cor}
\section{Conclusions}
The inequalities in Proposition \ref{eqn_bessel2}(iii)  and  in Corollary \ref{corodisRF}
are an extension to the rigged Hilbert spaces respectively of the inequalities (\ref{carabess}), and  (\ref{cararf}) for sequences in Hilbert spaces
 (see  \cite[Th. 2, Th. 3, Sec. 2]{young}). In the case of sequences, it follows immediately that: if $\{e_n\}_{n\in\mathbb N}$ is an orthonormal basis of $\H$, then
 $\{f_n\}_{n\in\mathbb N}$ is a Bessel sequence if, and only if, there exists a bounded operator $T:\H\rightarrow \H$ such that $f_n=T e_n$; $\{f_n\}_{n\in\mathbb N}$ is a Riesz-Fischer sequence if, and only if, there exists a bounded operator  $V:\H\rightarrow\H$ such that
 $V f_n = e_n$
(for frames and Riesz-bases see also \cite[Proposition 4.6]{balastoeva}). Since in the spaces of distributions  the orthonormality is not defined, a sort of ''orthonormal basis'' is played    by the \textit{Gel'fand basis}: see \cite{TTT} and \cite[Definition 5.3]{FT}. So, it would be appropriate to carry on a further study, started in \cite{TTT}, about the  transformations between Gal'fand basis, Bessel, Riesz-Fischer maps, distribution frames, and Riesz distribution basis.
\section*{Acknowledgments}
 This work has been realized within of the activities of Gruppo UMI Teoria del\-
 l'Approssimazione e Applicazioni and Gruppo Nazionale per l'Analisi Matematica, la Probabilit\`a e le loro Applicazioni  (GNAMPA) of the Istituto Nazionale di Alta Matematica (INdAM).
\bibliographystyle{amsplain}

\end{document}